\documentclass{amsart}
\usepackage[russian,english]{babel}
\usepackage{amsmath,amssymb,amsthm,amsfonts}
\usepackage[matrix,arrow]{xy}

\usepackage{dingbat}

\numberwithin{equation}{section}
\newcommand{\bC}{{\mathbb C}}

\newcommand{\bK}{{\mathbb K}}

\newcommand{\bN}{{\mathbb N}}

\newcommand{\bR}{{\mathbb R}}

\newcommand{\bZ}{{\mathbb Z}}

\newcommand{\tit}{\textit}
\newcommand{\sub}{\subset}
\newcommand{\wtilde}{\widetilde}
\newcommand{\bs}{\backslash}

\newcommand{\SH}{\operatorname{\mathit{SH}}}

\newcommand{\iso}{\cong}

\newtheorem{thm}{Theorem}[section]
\newtheorem{theorem}[thm]{Theorem}
\newtheorem{property}[thm]{Property}

\newtheorem{definition}[thm]{Definition}

\newtheorem{remark}[thm]{Remark}

\newtheorem{example}[thm]{Example}

\newtheorem{lemma}[thm]{Lemma}
\newtheorem*{claim}{Claim}

\begin{document}
\title[Non-standard Symplectic Structures via Symplectic Cohomology]{Non-standard Symplectic Structures via Symplectic Cohomology}
\author[Tran]{Dustin Tran}
\date{May 11, 2012}

\begin{abstract} We examine how symplectic cohomology may be used as an invariant on symplectic structures, and investigate the non-uniqueness of these structures on Liouville domains, a field which has seen much development in the past decade. Notably, we prove the existence of infinitely many non-standard symplectic structures on finite type Liouville manifolds for dimensions $n\geq 6$. To do this, we build up notions of Liouville domains, Lefschetz fibrations, and symplectic cohomology.\end{abstract}
\maketitle

\section{Introduction}

In 1985, Gromov \cite{G} illustrated a non-standard exact symplectic structure on $\bR^{2n}$ for $n>1$. However, the standard symplectic structure is known to be Liouville, whereas Gromov's exotic structure is not. The first exotic structure proven to be Liouville was not discovered until 2005, by Seidel-Smith \cite{S-S}, demonstrating its existence on $\bR^{2n}$ for any even $n>4$. This was later expanded upon by Mc-Lean \cite{M-L}, covering $\bR^{2n}$ for all $n>4$. Most recently, Abouzaid-Seidel \cite{A-S} proved the existence of infinitely many (distinct) exotic structures on smooth complex affine varieties of real dimension $n\geq6$. We note that by choosing a particular K\"{a}hler form, there exists a canonical correspondence between smooth complex affine varieties and finite type Liouville manifolds.

This paper seeks to reproduce the most recent result, though to a slightly more accessible audience. In light of being terse as well, this paper must also shorten remarks regarding any deeper results or more comprehensive theorems. That is, while Abouzaid-Seidel \cite{A-S} produces more than the above result, this paper sets to prove only that.

We use symplectic cohomology and build up its necessary tools, while offering references to more drawn-out, extensive arguments and their corresponding results. To shorten the proof remarkably, this paper also assumes several deep results and their corresponding properties. To be precise, we omit the proof of any result necessitating wrapped Floer cohomology, which would otherwise require extensive elaboration and introductions to more complicated techniques. Such proofs will be cited accordingly.

\section{Manifolds, Handles, and Fibrations}
\subsection{Liouville domains and manifolds.}
We start with some preliminary definitions on the classes of symplectic manifolds used in this paper. Recall that if a symplectic manifold $M^n$ is exact, then its symplectic form $\omega_M$ has a corresponding \tit{primitive} $\theta_M$. That is, there exists a $1$-form $\theta_M$ satisfying $\omega_M=d\theta_M$. The \tit{Liouville vector field} $Z_M$ is the vector field which is $\omega_M$-dual to $\theta_M$, i.e., $\omega_M(Z_M,\cdot)=\theta_M.$

\begin{definition}A Liouville domain is a compact exact symplectic manifold with boundary $M^{2n}$, such that $Z_M$ points strictly outwards along $\partial M$. \end{definition}

Given a manifold $M^{2n}$, $h:M\to\bR$ is an \tit{exhausting function} if it is bounded below and proper (continuous such that the inverse image of compact subsets is compact). An exact symplectic structure is \tit{complete} if the flow of $Z_M$ exists for all time.
\begin{definition}A Liouville manifold is a complete exact symplectic manifold $M^{2n}$ which admits an exhausting function $h:M\to\bR$ with the following property:

There exists a sequence $\{c_k\}_{k\in\bN}$ for $c_k\in\bR$, approaching $+\infty$ as $k\to\infty$, such that $dh(Z_M)>0$ along $h^{-1}(c_k)$.\end{definition}
We say that a Liouville manifold $M^{2n}$ has \tit{finite type} if $dh(Z_M)>0$ outside a compact subset $C\sub M$. Note that for any Liouville manifold $M^{2n}$, its corresponding sublevel sets yield an exhaustion of $M$ by Liouville domains.

As expected, any Liouville domain can be canonically extended to a finite type Liouville manifold by attaching an infinite cone to the boundary. Conversely, any finite type Liouville manifold can be truncated to a Liouville domain which is a sufficiently large sublevel set of the exhausting function.

\begin{example}\label{example1} Consider $(\bC^n,\omega_{std})$, where $\omega_{std}=\sum_j dx_j\wedge dy_j$, the $z_j=x_j+iy_j$ being the standard coordinates on $\bC^n$. Then the Liouville vector field $Z_{\bC^n}$ is $\sum r_j\frac{\partial }{\partial r_j}$, where $(r_j,\theta_j)$ denote the polar coordinates for $z_j$. It is easy to see that $\bC^n$ is a Liouville domain.\end{example}
\begin{example}\label{example2} Let $(M^n,\omega_{std})$ be a symplectic manifold where $\omega_{\text{std}}=d\lambda=\sum dp\wedge dq$ and $T^*M$ its cotangent bundle. Then $Z_M=\sum p\frac{\partial}{\partial p}$ generates the Liouville flow of "radial rescaling." Choosing a metric on $M$ makes the corresponding unit disc bundle a Liouville domain.  \end{example}

Now that we have three particular classes of symplectic manifolds to consider, the next question would be whether some notion of an isomorphism between two manifolds of these types exists.
\begin{definition}Two Liouville domains $M$, $\wtilde{M}$ are deformation equivalent if there exists a diffeomorphism $\phi:M\to \wtilde{M}$ and a one-parameter family of Liouville domain structures on $M$, which interpolate between $(\omega_M,\theta_M)$ and $(\phi^*\omega_{\wtilde{M}},\phi^*\theta_{\wtilde{M}})$.\end{definition}

\begin{definition} Let $M$ and $\wtilde{M}$ be Liouville manifolds. $\phi:M\to \wtilde{M}$ is an exact symplectomorphism if it is a diffeomorphism such that $\phi^*\theta_{\wtilde{M}}-\theta_M$ is an exact 1-form.\end{definition}
Strengthening the above notion, if $M$ and $\wtilde{M}$ are Liouville manifolds of finite type, we say that $\phi:M\to \wtilde{M}$ is a \tit{strictly exact symplectomorphism} if $\phi$ is a diffeomorphism and $\phi^*\theta_{\wtilde{M}}-\theta_M$ is not only an exact 1-form, but the derivative of a compactly supported function.

Note that each of these symplectomorphisms is also called a "Liouville isomorphism" by convention. Because we consider three distinct classes of manifolds that are all "Liouville," the term "Liouville isomorphism" is broad and less specific, depending on the context of the Liouville manifold considered.

\begin{example} Following Example \ref{example1}, all Liouville domains constructed in this manner are deformation equivalent, with corresponding completions symplectomorphic to $T^*M$ itself. \end{example}

If two Liouville domains are deformation equivalent, then their natural extensions to finite type Liouville manifolds are strictly exact symplectomorphic. The converse is true as well, when restricting Liouville manifolds of finite type to Liouville domains via the same choice.

\subsection{Weinstein handles.}
We now define a particular class of Liouville manifolds that will prove useful when extending Liouville domains to larger ones.
\begin{definition} A Weinstein manifold is a complete exact symplectic manifold, which admits an exhausting Morse function $h:M\to\bR$ with the following properties:
\begin{itemize}
\item At each critical point $x$ of $h$, $Z_M = 0$, and the quadratic form on $TM_x$ given by $X \mapsto D^2 h(X,DZ_M(X))$ is positive definite.
\item Outside the critical point set of $h$, $dh(Z_M) > 0$.
\end{itemize}\end{definition}

Readers interested in more information about Weinstein manifolds may consult \cite{W}, \cite{A}. We are only concerned with Weinstein handle attachment, i.e., attaching handles which satisfy the above Weinstein manifold conditions. Recall that for a manifold $M^n$ which admits a smooth embedding $\epsilon:S^{k-1} \hookrightarrow \partial M$, a \tit{$k$-handle attachment} is the process of gluing $H^k:=D^k\times D^{n-k}$ to the boundary $\partial M$, where the framing data defines the particular gluing construction.

\begin{property}\label{pr:2.5} Let $M$ be a Liouville domain and suppose there exists an embedding $\epsilon:S^{k-1} \hookrightarrow \partial M$ such that the pullback of $\theta_{M}$ to $S^{k-1}$ vanishes. Depending on a choice of framing data, we can attach a Weinstein $k$-handle to the boundary $\partial M$, which defines a larger Liouville domain $\wtilde{M} \supset M$.\end{property}

For example, suppose that a Liouville domain $M$ has two connected components $M_1,$ $M_2$. Taking a point on each boundary $\partial M_i$, the two points together form $S^{0}$, so $k=1$ and the embedding is clearly evident. We attach a Weinstein 1-handle to form a Liouville domain $\wtilde{M}\supset M$. This is called taking the \tit{boundary connect sum} of the two components, which will be useful later in the proof.

\subsection{Lefschetz fibrations.} We introduce another concept useful for constructing the non-standard symplectic structures. We later show that given a Lefschetz fibration $E$, we can construct another one $\wtilde{E}$ which is almost symplectomorphic to $E$ and its symplectic cohomology $SH^*(\wtilde E)$ vanishes (see Theorem \ref{th:3.1}). For a more comprehensive introduction to Lefschetz fibrations, one may consult \cite{M-S}, \cite{S}.
\begin{definition} Let $E^{2n+2}$ be an exact symplectic manifold. An exact (symplectic) Lefschetz fibration with fibre $M$ is a smooth map $\pi_E: E \rightarrow \bC$, which satisfies the following properties:\begin{itemize}\parskip1em

\item At each regular point, $\mathrm{ker}(D\pi_E) \subset TE$ is a symplectic subspace.

\item Lefschetz condition: At each critical point, it has a neighborhood $U$ such that our fibration is the map
\begin{equation} \label{eq3}\pi_E|_U:\bC^{n+1} \longrightarrow \bC\end{equation}
$$\quad z=(z_1,...,z_{n+1}) \longmapsto z_1^2 + \cdots + z_{n+1}^2 + constant.$$
\end{itemize}
\end{definition}
To be technical, since both the fibres and base are non-compact, we also need to impose another condition on Lefschetz fibrations, in order to control the geometry near infinity. However, since this is not a problem in our later constructions, we ignore the condition. Informally, a Lefschetz fibration can be thought of as a complex analogue of a Morse function.

\begin{definition}
A vanishing path on a Lefschetz fibration $\pi_E: E \to \bC$ is a properly embedded path $\beta: \bR^+ \to \bC$ such that $\beta(0)$ is a critical value of $\pi_E$ and $\beta(t)$ is regular value for all $t>0$. Moreover, for sufficiently large $t_0$ and some angle $a_\beta\in S^1\sub \bC$,
\begin{equation} \label{eq4}
\beta(t) = t a_\beta \text{ for all $t \geq t_0$}.\end{equation}\end{definition}
Given a vanishing path $\beta$, the \tit{Lefschetz thimble} $\Delta_\beta$ is a properly embedded Lagrangian submanifold of $E$ such that $\pi_E(\Delta_\beta)=\beta(\bR^+)$.

Suppose we are given a Lefschetz fibration $E^{2n+2}$ and a basis of vanishing paths $\{\gamma_1,\dots,\gamma_m\}$. Note that there are $m$ critical values corresponding to the $\gamma_k$'s,  and that each $\gamma_k\to\infty$ as $t\to\infty$. Now take the fibre $E_z$ for $z$ real and sufficiently large, and (canonically) identify it with $M$. Then the intersections of the Lefschetz thimbles $\{\Delta_{\gamma_1},...,\Delta_{\gamma_m}\}$ with $E_z$ give rise to a set of vanishing cycles $\{V_1,\dots,V_m\}$ in the fibre $M$, forming a basis in $M$. Thus, we have established the following construction: given a Lefschetz fibration with a basis of vanishing paths, we can construct a corresponding basis of vanishing cycles on each fibre.

\begin{definition}Given a manifold $M^{2n}$, a Lagrangian sphere $V$ is a Lagrangian submanifold $V\sub M$ (isotropic of dimension $n$), which is exact and admits a diffeomorphism $\phi_i:S^n \rightarrow V$ that is unique up to isotopy and composition with elements of $O(n+1)$. \end{definition}

Indeed, in the above construction, the corresponding vanishing cycles are Lagrangian spheres.

We are primarily concerned in the converse construction, with respect to Liouville domains. That is, let $M^{2n}$ be a Liouville domain and suppose there exists a set of Lagrangian spheres $\{V_1,...,V_m\}$. We can construct an exact Lefschetz fibration over the disc with fibre $M$, such that the $V_i$ form a basis of vanishing cycles (see \cite{S2} for the construction). The corresponding total space is a Liouville domain $E^{2n+2}$. Furthermore, $E$ comes with a set of Lefschetz thimbles $\{\Delta_1,\dots,\Delta_m\}$ that are determined by the set of Lagrangian spheres $\{V_1,\dots,V_m\}$.

\section{Symplectic Cohomology}
We now introduce the main invariant we will use on these three classes of manifolds, first considering Liouville domains. We denote $SH^*$ as short-hand for symplectic cohomology. For a more comprehensive introduction to symplectic cohomology, one may consult Seidel's survey article \cite{S}. Fix $\bK$ to be the desired field of coefficients.
\begin{definition}If $M$ is a Liouville domain of dimension $2n$, then the symplectic cohomology of $M$ with $\bK$-coefficients, $SH^*(M)$, is a $\bZ/2$-graded $\bK$-vector space with a natural $\bZ/2$-graded map
\begin{equation}H^{*+n}(M;\bK)\to SH^*(M).\end{equation}\end{definition}

Given two Liouville domains $U$, $M$ where $\dim U=\dim M$, and an embedding $\epsilon:U\to M$ such that $\epsilon^*\theta_M-c\theta_U$ is exact for some constant $c>0$, then Viterbo's construction \cite{V} assigns to each embedding a pull-back restriction map
\begin{equation}SH^*(\epsilon): SH^*(M)\to SH^*(U).\end{equation}
This is homotopy invariant within the space of all such embeddings, and functorial with respect to composition of embeddings. More generally, by a parametrization argument of Viterbo's construction, this is invariant under isotopies of embeddings, within the same class, so it must be invariant under deformation equivalent Liouville domains.

We now aim to understand symplectic cohomology on our two other classes of manifolds. In particular, with the natural correspondence between Liouville domains and Liouville manifolds of finite type, we can similarly define symplectic cohomology on the latter. And for the same reasons as invariance under deformation equivalent Liouville domains, symplectic cohomology must also be invariant under strictly exact symplectomorphisms (between finite type Liouville manifolds).

We aim to generalize this feature for Liouville manifolds of finite type to arbitrary Liouville manifolds.
\begin{definition}If $M$ is a Liouville manifold, then
\begin{equation}\label{eq2} SH^*(M)=\lim_{\leftarrow}SH^*(U),\end{equation}
where the limit is over all pairs $(U,k)$ such that the Liouville vector field $Z_M$ satisfying $\theta_M+dk$ points strictly outwards along $\partial U$, where $U\sub M$ is a compact codimension zero submanifold and $k$ is a function on $M$.  \end{definition}
By \eqref{eq2}, it is easy to see that $SH^*$ is invariant under exact symplectomorphisms between Liouville manifolds. And as an elementary consequence, $SH^*$ must also be invariant under exact symplectomorphisms between Liouville manifolds of finite type. In fact, this can be generalized even further.

\begin{property} \label{pr:2.12}$SH^*$ is invariant under symplectomorphisms between Liouville manifolds of finite type. \end{property}
\begin{proof}  Given a symplectomorphism $\psi:(M,\theta_M)\to(\wtilde{M},\theta_{\wtilde{M}})$ between two finite type Liouville manifolds, we aim to deform $\psi$ to an exact one.
Note that any finite type Liouville manifold has a cylindrical end. That is, there exists a compact set $C\sub M$ such that $M\bs C$ is exact symplectomorphic to $(A\times[0,\infty],r\alpha)$, where $r$ is the coordinate for $[0,\infty]$ and $\alpha$ is a contact form on $A:=\partial M$.

Now consider $\theta:= \psi^*\theta_{\wtilde{M}}-\theta_M$ and restrict the closed 1-form $\theta$ to the end $E=[0,\infty)\times A$. Then $\theta|_E$ can be represented as $\pi^*\theta|_A+dh$, where $\pi:E\to A$ is the projection and $h$ is a real-valued smooth function. Let $f:[0,\infty)\to[0,1]$ be a bump function which is $1$ near $0$ and $0$ on $[1,\infty)$, and
\begin{equation}\wtilde\theta:=\pi^*\theta|_A+d(f(s) h),\hspace{5mm} \wtilde h:= (1-f(s))h.\end{equation}

Let $V$ be the vector field dual to the closed 1-form $-\wtilde\theta $ with respect to $d\theta_M$. Then by evaluating $V$ on $[1,\infty]\times A$, it follows that $V$ must be represented as some $c\wtilde{V}$, where $\wtilde{V}$ is a vector field that does not depend on the constant $c>0$. Thus, $V$ integrates to an isotopy $\{\phi_t\}_{0\leq t\leq 1}: (M,\theta_M)\to(\wtilde M,\theta_{\wtilde{M}})$ such that $\phi_0=id$. Also note that the Lie derivative $L_V\theta_M=-\wtilde\theta+d(\theta_M(V))$. We conclude that
$$\phi:=(\psi\circ\phi_1)^*\theta_{\wtilde{M}}=\phi_1^*\theta_M+\phi_1^*\wtilde\theta + d\wtilde h\circ\phi_1=\theta_M+dg$$
for some smooth function $g$.
\end{proof}
\begin{remark} It is not known whether this is also true between Liouville manifolds in general. An analogous argument as above does not work, since the deformation works only when there there is a cylindrical end, i.e., at least one of the Liouville manifolds is of finite type. \end{remark}

We also note that there is a canonical map from ordinary cohomology to $SH^*$ compatible with the ring structures, providing symplectic cohomology with a large scale of useful properties. In particular, $SH^*$ must be a unital graded commutative ring.

\section{Lefschetz Fibration Reconstruction}
We now use our tools coming from Weinstein handle attachments and properties of Lefschetz fibrations to construct a Lefschetz fibration which is almost symplectomorphic to the original Lefschetz fibration and has vanishing symplectic cohomology groups.

Recall that for $M, \wtilde{M}$ symplectic manifolds, an \tit{almost symplectomorphism} is a diffeomorphism $\phi:M\to\wtilde{M}$ together with a one-parameter family $\{\omega_t\}_{0\leq t\leq1}$ of symplectic structures which interpolate between $\omega_0=\omega_M$ and $\omega_1=\phi^*\omega_{\wtilde{M}}.$

Let $E^{2n+2}$ be the total space of an exact Lefschetz fibration with fibre $M$ and vanishing cycles $\{V_1,\dots,V_m\}$. Further, suppose that inside $M$, there are Lagrangian discs $\{W_1,...,W_m\}$ such that for all $1\leq i,j\leq m$, $W_i$ intersects $V_j$ transversally and in a single point.

Now attach a Weinstein handle to the boundary of the first Lagrangian disc $\partial W_1$, obtaining a new Liouville domain $M'$, and
separate $\partial W_1$ into its two natural hemispheres. That is, decompose $\partial W_1$ into two codimension zero submanifolds $U_-,U_+$, which intersect on their common boundary $\partial U_+=\partial U_-$ and have distinct Euler characteristic values. Consider a smooth function $g: S^{n-1}\longrightarrow \bR$
such that
\begin{itemize}\item $g^{-1}(0)=\{p:p\in U_+\cap U_-\}$ and $0$ is a regular value,
 \item $g$ is strictly positive on the interior of $U_+$ and strictly negative on the interior of $U_-$.
 \end{itemize}
 Then $g$ extends to a smooth function $\wtilde g: \bR^n \longrightarrow \bR$ such that $\wtilde g(tp) = t^2 \wtilde g(p)$ for all $|p| \geq 1/2$ and $t \geq 1$. Now consider
 \begin{equation}\mathrm{Graph}(d\wtilde g)\sub T^*\bR^n\iso \bR^{2n},\end{equation} and intersect it with $D^{2n}$, defining the result as
$G:=\mathrm{Graph}(d\wtilde g)\cap D^{2n}\sub D^{2n}$.

\begin{claim}\label{le:4.1} $\partial G$ is disjoint from the submanifold $\{0\}^n\times S^{n-1}$ near where the handle is attached. \end{claim}
\begin{proof} (will prove this in final draft, but for now, see \cite{A-S} for details of the proof.)
\end{proof}

Then by the above claim, $\partial G\sub \partial M'$, so long as the size of the handle $\epsilon$ is chosen to be sufficiently small. This grants us another Lagrangian disc $W_1'=G$ in $M'$ with boundary $\partial W_1' = \partial G$. Now attach yet another Weinstein handle to $W_1'$, denoting the resulting manifold as $M''$. By construction, $M''$ contains a Lagrangian sphere obtained by gluing together $W_1'$ and the disc of the second handle. And we also have a Lagrangian disc $W_1''$ which is the disc of that second handle.

Intuitively, we are attaching a handle to $\partial M$ and then attaching a second handle to the boundary of the first handle, in a particularly nice way that allows us to hold several 'nice' properties.

Now we iterate this process for every $W_i$, for $1 \leq i \leq m$, denoting the final result as $\wtilde{M}\supset M$.
By construction, $\wtilde{M}$ holds a new set of vanishing cycles $\{\wtilde{V}_1,...,\wtilde{V}_{3m}\}$, which correspond to the new set of Lagrangian spheres $\{L_1,L_1',L_1'',...,L_m,L_m',L_m''\}.$
Using these vanishing cycles, we can now build a new Lefschetz fibration, whose total space will be denoted by $\wtilde{E}$, which has Lefschetz thimbles $(\wtilde{\Delta}_1,\dots,\wtilde{\Delta}_{3m})$.

\begin{theorem} \label{th:3.1}
When $E$ is of dimension $2n+2\geq 6$, the resulting $\wtilde{E}$ is almost symplectomorphic to $E$ and $\SH^*(\wtilde{E})=0$ on all coefficient fields $\bK$.
\end{theorem}

Readers interested in the proof may consult \cite{A-S}.

\begin{remark}The proof involves a number of clever applications and complicated techniques, such as using wrapped Floer cohomology, Dehn twists, and Hurwitz moves to reduce to simpler cases. This construction, credited to Abouzaid-Seidel \cite{A-S}, is known as homologous recombination. The name comes from its biology term, a type of genetic recombination that produces new nucelotide sequences from similar but not identical chromosomes. \end{remark}

\section{Symplectic Mapping Torus Reconstruction}
Before proving the main result, we will need to introduce the symplectic mapping torus to invoke another property we use in the proof.

Let $M^{2n}$ be a Liouville domain and $\phi: M \to M$ a diffeomorphism such that $\phi$ is the identity near $\partial M$ and $\phi^*\theta_M - \theta_M$ is the derivative of a function $f$ vanishing near the $\partial M$. Without loss of generality, we assume $|f(x)| < 1$ everywhere; otherwise, we rescale $\omega$ accordingly.
\begin{definition} The symplectic mapping torus $M_\phi$ is the manifold with corners
\begin{equation} \label{eq:4.1}
M_\phi:=\frac{[-1,1] \times \bR \times M}{(s,t,x) \sim (s,t-1,\phi(x))}. \end{equation}
It carries the $1$-form
\begin{equation} \label{eq:4.2}
\theta_{M_\phi}:=\theta_M + s\, dt + d(t\, f) = \theta_M + (s + k) dt + t \, df,
\end{equation}
which is invariant under the equivalence relation, and $\omega_{M_\phi}:=d(\theta_{M_\phi})=\omega_M + ds \wedge dt$. \end{definition}
We aim to construct a Liouville domain from the symplectic mapping torus $M_\phi.$ The Liouville vector field $Z_{M_\phi}$ is
\begin{equation} \label{4.3} Z_{M_\phi}:=Z_M + (s + f) \partial_s - t X_f,\end{equation}
where $X_f$ is the Hamiltonian vector field of $f$, i.e., $\omega_M(\cdot,X_f) = dX_f$. Note that this Liouville vector field points strictly outwards along the boundary face $\partial M$, and because $|f(x)|<1$ by assumption, $Z_{M_\phi}$ does the same on the boundary face $\{s = \pm 1\in [-1,1]\}$.

We now aim to round off the corners. By the collar neighborhood theorem (a variant of the tubular neighborhood theorem), there exists  a neighborhood $N \supset \partial M$ which admits a diffeomorphism $N \to [\frac{1}{2},1] \times \partial M$, denoting the first component of the map by $r: N \rightarrow [\frac{1}{2},1]$. Note that by construction, $\theta_M|_N$ agrees with the pullback of $r(\theta_M|\partial M)$. And without loss of generality, we assume $\phi = \mathrm{id}$ and $f = 0$ on $N$, since conjugating $\phi$ with the Liouville flow can always deform it. Now define
\begin{equation}
T := \{(s,t,x) \in [-1,1] \times \bR \times M \;:\; \text{either $x \notin N$ or $r(x) \leq \chi(s)$}\}/ \sim\},
\end{equation}
with same equivalence relation
\begin{equation}(s,t,x) \sim (s,t-1,\phi(x)),\end{equation}
and $\chi: [-1,1] \rightarrow [\frac{1}{2},1]$ is a function with the following properties:
\begin{itemize}
\item $\chi(\pm 1) = \frac{1}{2}$ and $\chi(0) = 1$,
\item $\chi$ is continuous everywhere, and smooth on $(-1,1)$,
\item $\chi' > 0$ on $(-1,0)$, $\chi' < 0$ on $(0,1)$, and $\chi''(0) < 0$,
\item Restricting $\chi$ to $[-1,0]$ or $[0,1]$, all derivatives of its inverse, $\chi^{-1}: [\frac{1}{2},1] \to [-1,0]$ or $[0,1]$ respectively, vanish at $t=\frac{1}{2}$.
\end{itemize}
The last-mentioned property of $\chi$ ensures that $\partial T$ is smooth. We equip $T$ with its canonical symplectic structure obtained from $M_\phi$. That is, we have
\begin{equation} \theta_T:= \theta_{M_\phi}|_T, \hspace{5mm}\omega_T:= \omega_{M_\phi}|_T, \hspace{5mm} Z_T:= Z_{M_\phi}|_T.\end{equation}
We note that $Z_T$ must point outwards along $\partial T$, so $T$ is our desired Liouville domain.

\begin{property} \label{pr:4.1} Let $\pi:T\to[-1,1]\times S^1$ be the projection onto the first two coordinates, $\pi(s,t,x)=(s,t).$ The summand $\SH^*(T)_0 \subset \SH^*(T)$ which corresponds to free loops lying in the kernel of $\pi_*: H_1(T) \rightarrow \bZ$ is nonzero and finite-dimensional in each degree.
\end{property}
Readers interested in the proof may consult \cite[Lemma 4.2]{A-S}, which uses wrapped Floer cohomology and closed Reeb orbit arguments.

For $n>1$, there is an isotropic embedded loop in $\partial T$ whose image under $\pi_*$ generates $H_1([-1,1] \times S^1)$ (see \cite{E-M}). By Property \ref{pr:2.5}, we can attach a Weinstein $2$-handle to that loop in order to form a larger Liouville domain $\wtilde{T}\supset T$.

\begin{lemma} \label{le:4.2}
Consider $\SH^*(\wtilde{T})$ with coefficients in a finite field of characteristic $p$. Then, the even degree part is a commutative ring in which the equation $x^p = x$ has a finite number $N \geq 2$ of solutions.
\end{lemma}

\begin{proof}
Any solution of $x^p = x$ in the even part of $\SH^*(T)$ must be contained in the degree zero part of the subspace $\SH^*(T)_0$. Property \ref{pr:4.1} states that this part is finite-dimensional, so there must only be finitely many such solutions. By Property \ref{pr:4.1}, $\SH^*(T)_0$ is nonzero, so there are at least two solutions: $0$ and $1$. Since this is for $n>1$, our handle attachment is subcritical (all critical points have index $< n$). And by \cite{C}, the Viterbo restriction map $\SH^*(\wtilde{T}) \rightarrow \SH^*(T)$ must be an isomorphism. Thus, the result must be true for $\wtilde{T}$.
\end{proof}

\section{Non-standard Symplectic Structures}
\begin{theorem} \label{th:4.4} Let $M$ be a Liouville domain of dimension $n\geq6$. Then there exists a sequence of finite type Liouville manifolds $\{\wtilde{M}_k\}_{k\in\bN}$ such that for each $k\in\bN$,\begin{itemize}
\item $\wtilde{M}_k$ is almost symplectomorphic to $M$,
\item $\wtilde{M}_k$ is not symplectomorphic to $\tilde{M}_j$ for $j\neq k$. \end{itemize}
\end{theorem}
\begin{proof} By Theorem \ref{th:3.1}, we can construct a Liouville domain $\wtilde{M}$ from $M$, such that $\wtilde{M}$ is almost symplectomorphic to $M$ and $SH^*(\wtilde{M})=0$.

Now for $k\in\bN$, take the boundary connect sum of $\wtilde{M}$ and $k$ copies of the domain $U$ considered above, and enlarge that to a Liouville manifold $\wtilde{M}_k$. This is still almost symplectomorphic to $M$. Furthermore, by \cite{C}, the Viterbo restriction map
\begin{equation}
\SH^*(\wtilde{M}_k) \to \bigoplus_{i=1}^k \SH^*(\wtilde{T}).
\end{equation}
must be an isomorphism. Now let $\bK$ be a finite field of characteristic $p$. Then by Lemma \ref{le:4.2}, the even part of $\SH^*(\wtilde{T})$ has $N \geq 2$ solutions of $x^p = x$, so $\SH^*(\wtilde{M}_k)$ has $N^k$ solutions. Consequently, for $k\neq l\in\bN$,
\begin{equation}SH^*(\wtilde{M}_k)\neq SH^*(\wtilde{M}_l),\end{equation} so by Property \ref{pr:2.12}, $\wtilde{M}_k$ is not symplectomorphic to $\wtilde{M}_l.$
\end{proof}

\begin{remark} There is a similar theorem on Liouville manifolds in general, with an analogous argument proven in \cite{A-S}, though it uses more assumptions.\end{remark}

\begin{remark}
As a final note, while symplectic cohomology is a powerful tool in distinguishing symplectic structures, other such tools and invariants are quite necessary, as symplectic cohomology may fail to distinguish certain exotic structures from the standard one. As reference to such an example, one may refer to \cite{H}.\end{remark}

\end{document}